\documentclass[]{llncs}
\usepackage[T1]{fontenc}
\usepackage{amssymb,amsfonts,amstext,amsmath} %, amsthm
\usepackage{graphicx}
\usepackage{hyperref}
\hypersetup{
   colorlinks=true,
   urlcolor=blue,    % color of external links
   linkcolor=black,  % color of toc, list of figs etc.
   citecolor=blue,   % color of links to bibliography
}
\begin{document}
\title{On G\"{o}del's  treatment of the undecidable in 1931 \thanks{This paper has been presented at the Round Table \emph{Revisiting G\"{o}del's Incompleteness Theorem} WCP Rome 2024.}}
\titlerunning{Ont G\"{o}del's  treatment of the undecidable in 1931}
\author{Paola Cattabriga}
\authorrunning{Paola Cattabriga}
%\titlethanks{}
%
\institute{University of Bologna \\  \medskip
\href{https://orcid.org/0000-0001-5260-2677}{ 0000-0001-5260-2677}}
\maketitle
%\nopagenumber
\begin{abstract}
In this article we discuss the proof  in the short unpublished paper appeared in the 3rd volume of G\"{o}del's Collected Works entitled "On undecidable sentences" (*1931?), which provides an introduction to G\"{o}del's 1931 ideas regarding the incompleteness of arithmetic. We analyze the meaning of the negation of the provability predicate, and how it is meant not to lead to vicious circle. We show how in fact in G\"{o}del's entire argument there is an omission regarding the cases of non-provability, which, once taken into consideration again, allow a completely different view of G\"{o}del's entire argument of incompleteness. Previous results of the author are applied to show that the definition of a contradiction is included in the argument of *1931?. 
Furthermore, an examination of the application of substitution in the well-known G\"{o}del formula  as a violation of uniqueness is also briefly presented, questioning its very derivation.
\end{abstract}
 %%%%%%%%%%%%%%%%%%%%%%%%%%%%%%%%%%%%%%%%%%%%%%%%%%%%%%%%%%%%

\subsection*{Introduction}
There is a short unpublished paper appeared in the 3rd volume of G\"{o}del's  Collected Works  entitled {\it On  undecidable sentences} (*1931?) \cite{godel3}.   Probably as an outline of an informative conference, it has the advantage of essentially exposing the ideas that led G\"{o}del to the incompleteness of arithmetic  \cite{godel1}. Reading it allows us to identify two points that stand out.

The first, at the beginning, where a notion of completeness emerges that explicitly refers to decidability, according to which a system is complete when it is defined with respect to the decision.
 \begin{quote}
... a formal system $S$ is called complete (in the syntactic sense) if every sentence $p$ expressible in the symbols of $S$ is decidable from the axioms of $S$, that is, either $p$ or not-$p$ is derivable in finitely many steps from the axioms by means of the rule of inference of the calculus of logic. \cite{godel3}
 \end{quote}
 A conception that opens the way to the later developments of what today we know as the theory of computability \cite{borger,bgg,davis,davis1,hermes,mendel}. G\"{o}del contribution is in this sense pioneering in providing completely new results. As we will show, for this reason it is currently possible to identify some primary omissions within G\"{o}del's demonstrative strategy, which when properly investigated invalidate the 1931 incompleteness argument in its entirety  \cite{godel1,godel2,godel3,godel4}.
 
 The second point concerns what  G\"{o}del calls a "{\it detour through metamathematics}".   Metamathematical concepts can be expressed as number-theoretic statements, whose expressions demonstrably lie between arithmetic formulas, composed through addition, multiplication, and logical connectives (not, or, and, all, there is), where "all" and "there is " are allowed to refer only to natural numbers.    Although this is a detour through an external path to the arithmetic of natural numbers, we still get an arithmetic class.
Keeping in mind these two different levels, the arithmetic and the metalinguistic one, it is precisely thanks to this deviation that  G\"{o}del brings down from the latter within the extensional level of the arithmetic of natural numbers a class defined by negation, according to typically antinomian methods.

 As is known, for this purpose G\"{o}del will builds step by step  the arithmetization of syntax, according to which a numerical code is effectively assigned not only to the symbols of the language, but also to the finite sequences of symbols and in turn to the finite sequences of these sequences  \cite{godel1,godel2}.
 *1931? offers just a simple introductory example  for the metamathematical relation 
 "The formula $a$ is derivable from the formula $b$ by the rules of inference" as corresponding to the relation $R$ between natural numbers: 
 \begin{quote}
 $R$ holds between the numbers $m$ and $n$ if and only if the formula with the number $m$ is derivable from the formula with number  $n$. \cite{godel3}
 \end{quote}
 Among such metamathematical relations there are sequences representing proofs, which are expressed by the well-known  predicate of provability, $Bew(x)$, that means "$x$ {\it is a provable formula }" \cite{godel1,godel2,godel3}.

 \medskip

 We briefly  present  *1931?'s original argument,  and in next section we will develop these two points. For the formal system $S$ as an extension of arithmetic,  see  \cite{godel1,godel3,godel4}.  As noted by G\"{o}del, in the following proof it is tacitly assumed that any sentence provable in $S$ is true \cite{godel3}.  We maintain this tacit assumption also in the following section, we cannot go  into the examination of the burdensome consequences here.
  
 \bigskip 
 
 \begin{quote}

 G\"{o}del's incipit is as follows.
 
 \medskip
 
  \begin{quotation} {\it 
 Every formal system with finitely many axioms that contains the arithmetic of the natural numbers is incomplete.} \cite{godel3}
  \end{quotation}
  
  \medskip
  
 Demonstration strategy \cite{godel3}.
  
  \medskip
  
 Let us consider the formulas with one free variable that are contained in the formal system $S$ and think of these as ordered in a sequence
 \begin{equation}
 \varphi_1 (x), \varphi_2 (x), \varphi_3 (x), \dots , \varphi_n (x), \dots 
 \end{equation} 
 The class $K$ is then defined as
 \begin{equation}
 n \in K \iff \neg Bew (\varphi_n (n))
 \end{equation}
 and the emphasis is placed again on the crucial point that this class $K$ can be proved to be arithmetic. Then from the hypothesis that $S$ is an extension of arithmetic, it follows that among the formulas in (1) there will be one $\varphi_k (x)$ coextensive with $K$ such that
 \begin{equation}
 \varphi_k (x) \iff x \in K
 \end{equation}
 and therefore such that
  \begin{equation}
 \varphi_k (n) \iff \neg Bew (\varphi_n (n))
 \end{equation}
 for any natural number $n$.  
 In particular hence for $n = k$
   \begin{equation}
 \varphi_k (k) \iff \neg Bew (\varphi_k (k)).
 \end{equation}
 So that, if $\vdash_S \varphi_k (k) $ then also  $\vdash_S \neg Bew(\varphi_k (k)) $. That is, $ \varphi_k (k)$ is provable and not provable in $S$.
 If on the other hand we had 
  $\vdash_S \neg \varphi_k (k) $ then  $\vdash_S  Bew(\varphi_k (k)) $, namely $ \varphi_k (k)$ is provable and not provable in $S$.
 
 Therefore, always under the hypothesis of the correctness of $S$, we would obtain that both $\varphi_k (k)$ and $\neg \varphi_k (k)$ are provable in $S$. Absurd.
 G\"{o}del's conclusion is that $\varphi_k (k)$ is formally undecidable in $S$.
 \end{quote}

\section{As about the detour through metamathematics}
From (1) to (5), the solvability and decidability of formulas are  very involved, there can be no firm understanding if we ignore this. We therefore expose those aspects of the theory of computability which probably could not yet be taken into consideration in G\"{o}del's time.  The following is also explanatory and clarifying  regarding the author's previous works   \cite{catta1,catta2,catta3,catta4}.

 We recall some fundamental definitions.
 
 A predicate $P$ is said to be primitive recursive 
 if and only if its
characteristic function $C_P$ is primitive recursive \cite[179-180]{mendel}. 
If $P$ is a
predicate of $n$ arguments, then the characteristic function $C_P$ is 
defined as follows
\begin{equation}\notag
 C_P(x_1,\dots,x_n) =
\begin{cases}
0& \text{if $P(x_1,\dots,x_n)$ is true,}\\
1& \text{if $P(x_1,\dots,x_n)$ is false.}
\end{cases} 
\end{equation} 

If the characteristic function $C_P$  of a predicate $P$ is primitive recursive, then  $C_P$  is computable, and therefore $P$ is decidable.
Let us then  recall the well-known definitions 45. and 
46. for the provability predicate in G\"{o}del 's 1931  \cite[162-171]{godel1}.

\medskip

\begin{math}  
\text{45. } xBy \equiv Bw(x)\;\&\; [l(x)]\, Gl\, x = y, \\
\text{$x$ is a PROOF of the FORMULA y.}
\end{math}

\medskip

\begin{math}  
\text{46. } \text{Bew}(x) \equiv (Ey) y B x ,\\
\text{$x$ is a PROVABLE FORMULA. }
\end{math}

\medskip

$xBy$ stands for "$x$ is the G\"{o}del number of a proof of the formula with G\"{o}del number y" and is primitive recursive \cite[170-171]{godel1}. Being primitive recursive it is also computable. Let $x$ be the number of a formula $\varphi$ and $y$ the number of $\psi$, then $xBy$ tells us that $\vdash \varphi \rightarrow \psi $ is verifiable in a finite number of steps. So in other words, taking any natural number it can be decomposed to check whether or not it is the G\"{o}del number of a proof of $\psi$ from $\varphi$. And it is possible to establish whether it really is with a set of tests, each of which occurs a finite time.

As is well known, however, the modal property of {\it being provable} represented by  $Bew(x)$ is not verifiable in this way, for  $Bew(x)$  is based on an assertion of existence. $(Ey) y B x$ says nothing about the actual verification procedures, it only symbolizes and affirms their existence.
So, if we are simply given a formula $\varphi$ that is a candidate to become a theorem, then we know nothing about whether there is an effective algorithm for finding how to prove it. We can simply examine the proofs, one at a time, to see if the formula is a theorem. But if it is not a theorem one cannot predict if and when the formula will arise. In terms of computability theory,  $(Ey) y B x$ can generate the list of all the theorems, even infinitely, but there may be no actual procedure to predict whether the candidate formula will appear or not. If it wasn't actually a theorem, it wouldn't appear in the list. 

Reflecting on this path from 45. to 46. we note that what is intended to be only a symbolic notation, the existential unlimited quantifier $(Ey)$, ends up restricting a priori  the field of  what is really computable, that is,  it clearly prevents any further investigations.  We will show how precisely at this point we can instead open a path towards a completely new field for the effective computation.

$\neg \text{Bew}(x)$ is $\neg (Ey) yBx$, i.e. that there is no proof $y$ of $x$. This could be understood, either that the formula with G\"{o}del number $x$ is not  a theorem, or that $(y)\neg yBx$. However, what can always be defined by recursion  ( see \cite{borger,bgg,davis,davis1,hermes,mendel}) allows us to explore the effective in another way.

\medskip

 Let us begin adding to the list of functions (relations) 1-45 in 
G\"{o}del's 1931 two new relations, 45.1 and 46.1, in terms of the 
preceding ones by the procedures given in Theorems I-IV \cite[158-163]{godel1}. 

\medskip

\begin{math}  
\text{45.1. } xWy \equiv Bw(x)\;\&\; [l(x)]\, Gl\, x = \text{Neg}(y),\\
\text{$x$ is a REFUTATION of the FORMULA y.}
\end{math}

\medskip

\begin{math}  
\text{46.1. } \text{Wid}(x) \equiv(Ez) z W x,\\
\text{$x$ is a REFUTABLE FORMULA. }
\end{math}

 \medskip

\noindent
 $\text{Wid}$ 
is the shortening for ``Widerlegung" and must not be 
mistaken with the notion defined by G\"{o}del in note 63 referring instead to 
``Widerspruchsfrei" \cite[192-193]{godel1}.

$ xWy $ like $xBy$ is primitive recursive, so  they are two decidable predicates.  Their characteristic functions are therefore computable. 
We can accordingly trace some logical connections between them as follows.

\medskip

Let us call the characteristic functions of $xBv$ 
 and 
 $xWv$ 
respectively $C_{B}(x,v)$ and
$C_{W}(x,v)$.
As $xBv$ and   
 $xWv$  are  
primitive recursive 
then also $C_{B}(x,v)$ and $C_{W}(x,v)$  are  primitive recursive.
From now on, for any 
expression $X$  we use $\ulcorner X \urcorner$ to denote the 
G\"{o}del number of $X$, and  $PA$ refers to the Peano Arithmetics. For the following lemmas see also \cite{catta1,catta2}.

 \medskip

\begin{lemma}\label{notboth}
For any natural number $n$ and for any formula $\varphi$ in {\it PA} not both
$nW \ulcorner \varphi \urcorner$ 
and $nB \ulcorner \varphi \urcorner$.
\end{lemma}
\begin{proof}
Let us suppose to have both $nW \ulcorner \varphi \urcorner$ 
and $nB \ulcorner \varphi \urcorner$.
We should have then
$$Bw(n)\;\&\; [l(n)]\, Gl\, n = \ulcorner \varphi \urcorner \quad
\text{and} \quad
Bw(n)\;\&\; [l(n)]\, Gl\, n = \text{Neg}(\ulcorner \varphi \urcorner)$$
which is impossible  because no AXIOM  belongs to the system 
together with its 
negation and  the IMMEDIATE CONSEQUENCE    preserve
logical validity. For a detailed proof see \cite{catta1}. \qed
\end{proof}

\medskip
 
\begin{lemma}\label{complete1}
For any formula $\varphi$, and $n$ as the   G\"{o}del number of a proof
in {\it PA} of $\varphi$
$$ 
\vdash _{PA}\: C_B(n,\ulcorner 
\varphi\urcorner) = 0
 \;\;\&\;\; 
C_W(n,\ulcorner \varphi\urcorner) = 1
$$
\end {lemma}

\begin{proof}
As we can notice the two conjuncts are true: since $n$ is the   
G\"{o}del number of a proof
 of $\varphi$,
$C_B(n,\ulcorner 
\varphi\urcorner)= 0$ is true. 
By Lemma (\ref{notboth}) $n W \ulcorner 
\varphi\urcorner$ does not hold,
therefore it is true that
 $n$ is not the   G\"{o}del number of a refutation
 of $\varphi$,  $C_W(n,\ulcorner \varphi\urcorner)= 1$. \qed
\end{proof}

\medskip

\begin{lemma}\label{complete2}
For any formula $\varphi$, and $n$ as the   G\"{o}del number of a 
refutation
in {\it PA} of $\varphi$
$$ 
\vdash _{PA}\: C_W(n,\ulcorner 
\varphi\urcorner) = 0
 \;\;\&\;\; 
C_B(n,\ulcorner \varphi\urcorner)= 1
$$
\end {lemma} 
\begin{proof}
The two conjuncts are true: since $n$ is the   
G\"{o}del 
number of a refutation
of $\varphi$, 
$C_W(n,\ulcorner 
\varphi\urcorner)= 0$ is true. 
By Lemma (\ref{notboth}) $n B \ulcorner 
\varphi\urcorner$ does not hold,
therefore it is true that
 $n$ is not the   G\"{o}del number of a proof
 of $\varphi$, $C_B(n,\ulcorner \varphi\urcorner)= 1 $. \qed
\end{proof}

\medskip

\begin{lemma}\label{antidiag1}
For  any formula $\varphi$  in {\it PA}
 
\begin{itemize}
	\item[(i)]\it{ not both }  $$\vdash _{PA}\:  n B \ulcorner 
\varphi\urcorner \:\: \vdash _{PA} \: nW \ulcorner \varphi\urcorner,$$
    \item[(ii)] \it{for } $n$  \it{ as the   G\"{o}del number of a refutation
in  PA of } $\varphi$ $$ \vdash _{PA} \: nW \ulcorner \varphi\urcorner \Longleftrightarrow 
\neg  n B \ulcorner \varphi\urcorner ,$$
	\item[(iii)]\it{for } $n$  \it{ as the   G\"{o}del number of a proof 
	in  PA of } $\varphi  $
$$\vdash _{PA} \:   n B \ulcorner \varphi\urcorner 
\Longleftrightarrow \neg nW \ulcorner \varphi\urcorner.$$
\end{itemize}
\end{lemma}
\begin{proof}
$ $

(i) Immediately by Lemma (\ref{notboth}).

(ii) Let us assume $\vdash _{PA} 
n W \ulcorner \varphi\urcorner$, then
Lemma (\ref{complete2}) yields $\vdash _{PA} 
C_B(n,\ulcorner\varphi\urcorner)= 1$.
Hence by definition  $n B \ulcorner 
\varphi\urcorner$ is false,
consequently $\vdash _{PA}  \neg \: n B \ulcorner\varphi\urcorner$.
Conversely let us assume 
$\vdash _{PA}  \neg \: n B \ulcorner\varphi\urcorner$
then $ n B \ulcorner\varphi\urcorner$ is false 
and by Lemma
(\ref{complete2}) we attain $\vdash _{PA} 
 n W \ulcorner\varphi\urcorner$.

(iii) Let us assume $\vdash _{PA} 
n B \ulcorner \varphi\urcorner$, then
Lemma (\ref{complete1}) yields $\vdash _{PA}  C_W(n,\ulcorner\varphi\urcorner)= 1$.
Hence by definition  $n W \ulcorner\varphi\urcorner$ is false,
consequently $\vdash _{PA}  \neg \: n W \ulcorner\varphi\urcorner$.
Conversely let us assume 
$\vdash _{PA} \neg \: n W \ulcorner\varphi\urcorner$
then $ n W \ulcorner\varphi\urcorner$ is false 
and by Lemma
(\ref{complete1}) we attain 
$\vdash _{PA}   n B \ulcorner\varphi\urcorner$. \qed
\end{proof}

 \bigskip

 Let us now take up G\"{o}del's definition in *1931? according to which $Bew (x)$ means "{\it x is a provable formula}".  What exactly does the negation of  $Bew (x)$ mean from a logical point of view? And what  did it mean  for G\"{o}del himself?  It would appear that $\neg Bew (x)$ means "{\it x is not a provable formula}". We note that in *1931? this meaning is used immediately after (5) (see (4) in \cite{godel3}). So "non-provable" seems to be the correct meaning.
  
 \medskip 
 But if so we would be led to believe that
 
 \medskip
 
\begin{equation} \tag{I}\label{magga}
\text{ if  } \quad \vdash  \neg Bew (\varphi_n (n)) \quad \text{ then } \quad \vdash   \neg \varphi_n (n),
 \end{equation}
so we could directly conclude from (5) in G\"{o}del's  *1931? that
 \begin{equation} \tag{$\mathbf C$}\label{camarra}
\varphi_k (k)  \iff   \neg \varphi_k (k).
 \end{equation}
 
  \medskip
 
This would make it clear that the deviation through metamathematics is in fact just a way to introduce new contradictions by definition. In fact, a vicious circle would be obtained. But considering note 15 in \cite[150-151]{godel1} in G\"{o}del's  whole argument there is not a vicious circle. This is only possible  if  we think  that $ \neg Bew (x)$ simply means that {\it the G\"{o}del number $x$ is not the G\"{o}del number of a provable formula}. 
The formula $\neg Bew (\varphi_k (k))$, in fact, does not speak directly about its provability, but rather about the arithmetic property encoding the syntactic property of  "being a theorem" and its G\"{o}del number.
Therefore, footnote 15 suggests that the argument $\varphi_k (k)$ is not a proposition that asserts by itself  "I am not provable" (in which case there would be circularity, as in (\ref{camarra})), but is rather a formula that states "{\it my G\"{o}del number is not the G\"{o}del number of a provable formula}".
Therefore any presumed circularity is blocked thanks to coding via G\"{o}del-numbering. If this is the correct interpretation of G\"{o}del's thought then (I) is not really included in *1931?'s reasoning from (1) to (5).  

 \medskip
 
We are then leaded to notice that the entire reasoning lacks something that is left neglected and unconsidered.
 Instead of (I) it would be more logical to say that 
 
 \medskip
 
 \begin{equation}\tag{II}\label{urro}
\text{ if } \quad \vdash   Bew (\varphi_n (n)) \quad  \text{ then } \quad \vdash   \varphi_n (n),
 \end{equation}
and, in accordance,
  \begin{equation}\tag{III}\label{terro}
\text{ if }  \quad  \vdash \neg  Bew (\varphi_n (n))  \quad \text{ then }   \quad \nvdash   \varphi_n (n),
 \end{equation}
 
 \medskip
 
 \noindent where the consequent $ \nvdash   \varphi_n (n)$ more adequately expresses  that $ \varphi_n (n)$ is  not a provable formula.
 
 \medskip
 
 We can then ask: is there a way to show (\ref{magga}) and therefore  (\ref{camarra}) as included in *1931?'s argument  from (1) to (5)?
 
 \medskip
 
 The arithmetization of syntax proceeds recursively. We must therefore bring out through arithmetization what has remained neglected, namely how to derive  (\ref{magga}) and  (\ref{camarra}) through the detour of G\"{o}del theoretical numbering.
This is only possible thanks to our definition 45.1, 46.1, which give to  refutability the same recursive status as provability. 
We are thus able to  correctly express and include in the overall argument the other two cases neglected by G\"{o}del's argument.
We have therefore

\medskip

 \begin{equation}\tag{IV}\label{urroneg}
\text{ if }  \vdash   Wid (\varphi_n (n)) \quad \text{ then } \quad \vdash \neg  \varphi_n (n),
 \end{equation}
 and 
  \begin{equation}\tag{V}\label{terroneg}
\text{ if  }  \vdash \neg  Wid (\varphi_n (n)) \quad \text{ then } \quad  \nvdash  \neg  \varphi_n (n).
 \end{equation}

\medskip

Thanks to 45.1, 46.1  and our four previous lemmas we can now reach a correct proof of (\ref{magga}).
  Let us reconsider accordingly the definition (2) of the class $K$ in *1931? \cite{godel3}.
 
 \medskip
 
The following equivalences hold in {\it PA}.
\begin{xalignat}{3} 
  n \in K &\iff  \neg Bew (\varphi_n (n)) && \text{by definition } (2) \\ 
 & \iff\neg Ey \, yB(\varphi_n (n)) && \text{by definition 46.} \\ 
 & \iff (y)\, \neg yB(\varphi_n (n)) && \text{abbreviation \cite[52]{mendel} } \\ 
 & \iff \neg\, mB(\varphi_n (n)) && \text{particularization rule \cite[76]{mendel} }\\ 
  & \iff mW(\varphi_n (n)) && \text{lemma 4 (ii) }\\ 
  & \iff Wid(\varphi_n (n)) && \text{definition 46.1 }\\ 
 & \iff \vdash \neg \varphi_n (n) && \text{(\ref{urroneg})}
   \end{xalignat}

We can now also  state (I), and that the definition of $K$ is equivalent to 
\begin{equation}\label{alba}
n \in K \iff  \vdash \neg \varphi_n (n).
\end{equation}

 The supposition  that among the formulas in the list  (1) there will be one $\varphi_k (x)$ coextensive with $K$ such that
 $
 \varphi_k (x) \iff x \in K
$
yields
\begin{equation}\label{vicious}
 \varphi_k (n) \iff \neg \varphi_n (n) 
 \end{equation}
 and for $n = k $ we obtain  (\ref{camarra})
$$
 \varphi_k (k) \iff \neg \varphi_k (k).
$$

  \medskip
  
 This clearly shows that G\"{o}del's definition (2) is just introducing a contradiction. Apparently this definition is very similar to a creative axiom \cite[153-155]{suppes}, and it echoes the well known words by J. E.  Poincar\'{e}: logistics is no longer sterile, it generates contradictions. 
 
   \medskip

 Regarding the Criterions of the Theory of Definition which states the rules for proper definitions in mathematics \cite[151-173]{suppes}\cite{rogers}, we can add that the assumption $n = k $  in (5) somewhat resembles a  violation of the uniqueness (see also \cite{catta3,catta5,catta6,catta7}). Since the theory is the arithmetic of natural numbers, it is  assumed to include further the identity law $ x = y \rightarrow  ( \phi x \iff \phi y )  $. This law can be written also as  $ \neg ( \phi x \iff \phi y ) \rightarrow   \neg (x = y)$,  and 
 $( \phi x \iff  \neg  \phi y ) \rightarrow  x \neq y$.

So from (\ref{vicious}) and 
\begin{equation}
 ( \varphi_k (n)  \iff  \neg   \varphi_n (n) ) \rightarrow  k \neq n.  
 \end{equation}
we obtain that 
\begin{equation} k \neq n.  \end{equation}
Consequently the assumption that $k = n$ cannot be applied as a substitution to obtain (5), and,  the entire argument of G\"{o}del's *1931? is  meaningless.

 \bigskip
 
 \bigskip

\bibliographystyle{plain}

%%%%%%%%%%%%%%%%%%%%%%%%%%%%%%%%%%%%%

%
\end{document}